\documentclass[10pt]{amsart}
\usepackage{enumerate,amsmath,amssymb,latexsym,
amsfonts, amsthm, amscd}

\theoremstyle{plain}
\newtheorem{theorem}{Theorem}

\numberwithin{equation}{section}

\newcommand{\ra}{\rightarrow}

\newcommand{\meet}{\wedge}
\newcommand{\join}{\vee}

\begin{document}

\title {Distributive Quotients}

\date{}

\author[P.L. Robinson]{P.L. Robinson}

\address{Department of Mathematics \\ University of Florida \\ Gainesville FL 32611  USA }

\email[]{paulr@ufl.edu}

 \keywords{}

\begin{abstract}

We note that each lattice $L$ has a unique largest distributive quotient, of which every distributive quotient of $L$ is itself a quotient.

\end{abstract}

\maketitle

\bigbreak

Let $L$ be a lattice, with meet and join denoted by $\meet$ and $\join$ respectively. A {\it congruence} on $L$ is an equivalence relation $\theta$ on $L$ that is compatible with meet and join in the sense 
	\[(a_1, a_2) \in \theta \; {\rm and} \; (b_1, b_2) \in \theta \Rightarrow (a_1 \meet b_1, a_2 \meet b_2) \in \theta \; {\rm and} \; (a_1 \join b_1, a_2 \join b_2) \in \theta
\]
whenever $a_1, b_1, a_2, b_2 \in L$; the corresponding quotient lattice is the set $L/\theta$ of blocks (or equivalence classes) with meet and join well-defined by
	\[[a]_{\theta} \meet [b]_{\theta} = [a \meet b]_{\theta}, \; [a]_{\theta} \join [b]_{\theta} = [a \join b]_{\theta}
\]
whenever $a, b \in L$. A lattice $L$ is {\it distributive} (notation: $L \in \mathbb{D}$) exactly when it satisfies either (hence each) of the equivalent conditions 
	\[(\forall a, b, c \in L) \; \; a \meet (b \join c) = (a \meet b) \join (a \meet c), 
\]
  \[(\forall a, b, c \in L) \; \; a \join (b \meet c) = (a \join b) \meet (a \join c). 
\]
The purpose of this brief note is to record certain elementary facts regarding the distributive quotients of an arbitrary lattice: in particular, the fact that each lattice $L$ has a unique largest distributive quotient, of which every distributive quotient of $L$ is itself a quotient. 

\medbreak 

The congruences on $L$ themselves constitute a lattice $\Theta (L) = {\rm Con} (L)$ in which meet is intersection and the join of two congruences is the transitive closure of their union. Within $\Theta (L)$ we single out those congruences of $L$ relative to which the quotient is distributive: 
	\[\Theta_{\mathbb{D}} (L) = \{ \theta \in \Theta (L) : L/\theta \in \mathbb{D} \}. 
\]

\medbreak 

Note that $\Theta_{\mathbb{D}} (L)$ is a filter in $\Theta (L)$: that is, an up-set closed under finite intersections. In fact, more is true.

\begin{theorem} 
$\Theta_{\mathbb{D}} (L) \subseteq \Theta (L)$ is an up-set that is closed under arbitrary intersections. 
\end{theorem} 

\begin{proof} 
$\Theta_{\mathbb{D}} (L)$ is an up-set: if $\Theta_{\mathbb{D}} (L) \ni \theta_0 \subseteq \theta \in \Theta (L)$ then the inclusion $\theta_0 \subseteq \theta$ induces a surjective homomorphism $L/\theta_0 \ra L/\theta$ realizing $L/\theta$ as a quotient of the distributive lattice $L/\theta_0$; thus $L/\theta \in \mathbb{D}$ and so $\theta \in \Theta_{\mathbb{D}} (L)$. $\Theta_{\mathbb{D}} (L)$ is closed under arbitrary intersections: if  $\theta_{\lambda} \in \Theta_{\mathbb{D}} (L)$ for each $\lambda \in \Lambda$ then the canonical map 
	\[L \ra \prod_{\lambda} (L/\theta_{\lambda}) : a \mapsto ([a]_{\theta_{\lambda}})_{\lambda}
\]
factors through an injective homomorphism 
	\[L/(\cap_{\lambda} \theta_{\lambda}) \ra \prod_{\lambda} (L/\theta_{\lambda}) \in \mathbb{D};
\]
thus $L/(\cap_{\lambda} \theta_{\lambda}) \in \mathbb{D}$ and so $\cap_{\lambda} \theta_{\lambda} \in \Theta_{\mathbb{D}} (L)$. 
\end{proof}

\medbreak 

Thus, $\Theta_{\mathbb{D}} (L)$ has as least element its infimum
	\[\delta_L = \bigwedge \Theta_{\mathbb{D}} (L) = \bigcap \{ \theta : \theta \in \Theta_{\mathbb{D}} (L) \} \in \Theta_{\mathbb{D}} (L)
\]
and so $\Theta_{\mathbb{D}} (L)$ is principal with $\delta_L$ as generator: 
	\[\Theta_{\mathbb{D}} (L) = \: \uparrow \delta_L = \{ \theta \in \Theta (L) : \theta \supseteq \delta_L \}. 
\]
Observe that $L/\delta_L$ is the largest distributive quotient of $L$: in fact, if $L/\theta$ is any distributive quotient of $L$ then $\theta \in \Theta_{\mathbb{D}} (L) = \: \uparrow \delta_L$; thus, $\theta$ contains $\delta_L$ and so there is a canonical surjective homomorphism $L/\delta_L \ra L/\theta$. 

\medbreak

We may identify the generator $\delta$ for an arbitrary quotient as follows. 

\begin{theorem} 
If $\theta \in \Theta (L)$ then $\delta (L/\theta) = \delta_L \join \theta/\theta$. 
\end{theorem} 

\begin{proof} 
For convenience, write $\delta = \delta_L$. On the one hand, the isomorphism 
	\[(L/\theta)/(\delta \join \theta/\theta) \equiv L/(\delta \join \theta) \in \mathbb{D}
\]
places the congruence $\delta \join \theta/\theta$ in $\Theta_{\mathbb{D}} (L/\theta)$. On the other hand, let $\phi$ be a congruence of $L$ containing $\theta$: if $\phi/\theta \in \Theta_{\mathbb{D}} (L/\theta)$ then the isomorphism 
	\[L/\phi \equiv (L/\theta)/(\phi/\theta) \in \mathbb{D}
\]
forces $\phi \in \Theta_{\mathbb{D}} (L)$ so $\phi$ also contains $\delta$ and $\phi \in \: \uparrow (\delta \join \theta)$. Conclusion: $\Theta_{\mathbb{D}} (L/\theta) = \: \uparrow (\delta \join \theta/\theta)$. 
\end{proof}  

\medbreak 

We may identify the generator $\delta$ for a finite product as follows. 

\begin{theorem} 
$\delta (L_1 \times L_2) = \delta_{L_1} \times \delta_{L_2}$. 
\end{theorem}

\begin{proof} 
For convenience, write $\delta_1 = \delta_{L_1}$ and $\delta_2 = \delta_{L_2}$. On the one hand, the isomorphism 
	\[(L_1 \times L_2)/(\delta_1 \times \delta_2) \equiv (L_1/\delta_1) \times (L_2/\delta_2) \in \mathbb{D}
\]
places $\delta_1 \times \delta_2$ in $\Theta_{\mathbb{D}} (L_1 \times L_2)$. On the other hand, each $\theta \in \Theta(L_1 \times L_2)$ has the form $\theta_1 \times \theta_2$ for $\theta_1 \in \Theta (L_1)$ and $\theta_2 \in \Theta (L_2)$; now, if $\theta \in \Theta_{\mathbb{D}} (L_1 \times L_2)$ then 
	\[(L_1/\theta_1) \times (L_2/\theta_2) \equiv (L_1 \times L_2)/\theta \in \mathbb{D}
\]
forces $(L_1/\theta_1) \in \mathbb{D}$ and $(L_2/\theta_2) \in \mathbb{D}$ so that $\theta_1 \in \Theta_{\mathbb{D}} (L_1) = \: \uparrow \delta_1$ and $\theta_2 \in \Theta_{\mathbb{D}} (L_2) = \: \uparrow \delta_2$ whence $\theta = \theta_1 \times \theta_2 \in \: \uparrow (\delta_1 \times \delta_2)$. Conclusion: $\Theta_{\mathbb{D}} (L_1 \times L_2) = \: \uparrow (\delta_1 \times \delta_2)$. 
\end{proof} 

\medbreak 

Let us identify the generator $\delta_L$ of $\Theta_{\mathbb{D}} (L)$ for a lattice $L$ in some basic examples.  

\medbreak  

\noindent 
{\bf Example 0}. If $L$ is a distributive lattice then, as each quotient of $L$ is distributive, $\Theta_{\mathbb{D}} (L) = \Theta (L)$ and $\delta_L = \underline{0}$ is the equality (or diagonal) relation on $L$.

\medbreak 

\noindent 
{\bf Example 1}. The non-distributive `diamond' $M_3$ is simple; it follows at once that $\Theta_{\mathbb{D}} (M_3) = \{ \underline{1}\}$ so that $\delta_{M_3} = \underline{1} = M_3 \times M_3$ is the trivial congruence that fully collapses $M_3$. 

\medbreak 

\noindent 
{\bf Example 2}. The non-modular `pentagon' $N_5 = \{ 0, a, b, c, 1\}$ with $a > b$ yields a distributive quotient as soon as $a$ and $b$ are identified. Accordingly, $\delta_{N_5}$ is the principal congruence $\theta(a, b)$: that is, the smallest congruence containing the pair $(a, b)$; its only nontrivial block is the doubleton $\{a, b\}$.   

\medbreak 

\noindent
{\bf Example 3}. The free modular lattice $F_{\mathbb{M}} (3)$ on three generators $x, y, z$ admits a unique homomorphism to the free distributive lattice $F_{\mathbb{D}} (3)$ on $x, y, z$ respecting the generators; the kernel of this homomorphism is the principal congruence $\theta (u, v)$ that identifies $u = (y \join z) \meet (z \join x) \meet (x \join y)$ and $v = (y \meet z) \join (z \meet x) \join (x \meet y)$. As no smaller congruence can yield a distributive quotient, $\delta_{F_{\mathbb{M}} (3)} = \theta (u, v)$. The nontrivial blocks of this congruence are six doubletons and the diamond with top $u$ and bottom $v$. 

\medbreak 

\noindent
{\bf Example 4}. The case of the free lattice $F(n)$ on $n$ generators is similar: $\delta_{F(n)}$ is the kernel of the unique homomorphism $F(n) \ra F_{\mathbb{D}} (n)$ that respects all $n$ generators. 
 
\medbreak 

\medbreak 

\noindent
{\bf Remark 1}. We have considered only the class $\mathbb{D}$ of distributive lattices, but entirely similar considerations apply to the class $\mathbb{M}$ of modular lattices. Indeed, they apply to any equational class $\mathbb{K}$ of lattices: as $\mathbb{K}$ is closed under the formation of quotients, products and sublattices, if $L$ is any lattice then the filter
	\[\Theta_{\mathbb{K}}(L) = \{ \theta \in \Theta (L) : L/\theta \in \mathbb{K} \}
\]
is principal with generator 
	\[\kappa_L = \bigwedge \Theta_{\mathbb{K}} (L) = \bigcap \{ \theta : \theta \in \Theta_{\mathbb{K}} (L) \} \in \Theta_{\mathbb{K}} (L). 
\]
In particular, $L$ has a largest quotient $L/\kappa_L$ in any equational class $\mathbb{K}$ and each quotient of $L$ in $\mathbb{K}$ is actually a quotient of $L/\kappa_L$.   

\medbreak 

\noindent
{\bf Remark 2}. This offers a perspective on the free lattice $F_{\mathbb{K}} (P)$ over the class $\mathbb{K}$ generated by the poset $P$ as discussed in section 5 of [1]: thus, let $F(P)$ be the free lattice generated by $P$ as in [1] Corollary 5.7; the lattice $F_{\mathbb{K}} (P)$ arises as the quotient of $F(P)$ modulo the congruence $\kappa_{F(P)}$ when the elements of $P \subseteq F(P)$ lie in different blocks of $\kappa_{F(P)}$. 

\section{References}

\noindent 
[1] G. Gr$\Ddot {\rm a}$tzer, {\it Lattice Theory: First Concepts and Distributive Lattices}, W.H. Freeman and Company (1971); Dover Publications (2009). 

\end{document}